\documentclass[12pt]{amsart}
 
\usepackage{amssymb}
\usepackage{mathrsfs}
\usepackage{amsthm}

\newcommand{\N}{\mathbb{N}}
\newcommand{\M}{\mathscr{M}}

\newtheorem{thm}[equation]{Theorem}

\usepackage[margin=2cm]{geometry}
\begin{document}

\title{A short and elementary proof of Hanner's theorem}
\author{Aasa Feragen}

\address{University of Copenhagen, Department of Computer Science (DIKU), Universitetsparken 1, K-2100 Copenhagen Ø, Denmark}
\email{aasa@diku.dk}

\begin{abstract}
Hanner's theorem is a classical theorem in the theory of retracts and extensors in topological spaces, which states that a local ANE is an ANE. While Hanner's original proof of the theorem is quite simple for separable spaces, it is rather involved for the general case. We provide a proof which is not only short, but also elementary, relying only on well-known classical point-set topology.
\end{abstract}

\maketitle

\section{Introduction}

Denote by $\M$ the class of metrizable spaces. Hanner's theorem is a fundamental theorem in the theory of extensors and retracts, stating that a space which admits an open covering by ANEs for $\M$, is an ANE for $\M$. 

Proving that a space with a countable covering by open ANEs is an ANE is not hard \cite{hanner1}, but the original proof of the Hanner's general theorem is rather complicated \cite{hu}. We give a short and elementary proof, based on reducing the uncountable covering by ANEs to a countable covering by ANEs using a technique originating with J.~Milnor \cite{palais1}. Another short proof of the theorem has been given by J.~Dydak \cite{dydak} as part of his framework for the extension dimension theory.

\section{Preliminaries}

A metrizable space $Y$ is said to be an ANE for $\M$ if, given any space $X \in \M$ and any continuous map $f \colon A \to Y$ where $A$ is a closed subset of $X$, there exists a neighborhood $U$ of $A$ in $X$ and a continuous extension $F \colon U \to Y$ of $f$.

\begin{thm} \label{thm1}
\begin{itemize}
\item[i)] Any open subset of a space which is an ANE for $\M$ is an ANE for $\M$.
\item[ii)] If $X=\bigcup_{i \in I} U_i$ where the $U_i$ are disjoint open subsets of $X$ which are ANEs for $\M$, then $X$ is an ANE for $\M$.
\item[iii)] If $X=\bigcup_{n \in \N} U_n$ where the $U_n$ are open subsets of $X$ which are ANEs for $\M$, then $X$ is an ANE for $\M$.
\end{itemize}
\end{thm}

\begin{proof}
Claim \emph{i)} is trivial, and proofs of claims \emph{ii)}, and \emph{iii)} can be found in Hanner's article \cite{hanner1}.
\end{proof}

\section{Hanner's general theorem}

Our theorem is the following

\begin{thm}
If $X \in \M$ and $X=\bigcup_{i \in I} U_i$ where the $U_i$ are open subsets of $X$ which are ANEs for $\M$, then $X$ is an ANE for $\M$.
\end{thm}

\begin{proof}
Find a partition of unity $\{\varphi_i \colon X \to [0,1]\}_{i \in I}$ which is subordinate to the covering $\{U_i\}_{i \in I}$ of $X$. For each finite subset $T \subset I$ we denote
$$
W(T)=\{x \in X|\varphi_i(x)>\varphi_j(x) \ \forall \ i \in T, \ \forall \ j \in I \setminus T\}.
$$
This set is open because $W(T)=u_T^{-1}(0,1]$ for the continuous map
$$
u_T \colon X \to [0,1], \quad u_T(x)=\max \{0, \min \{\varphi_i(x)-\varphi_j(x)|i \in T,\ j \in I \setminus T\}\}.
$$
Furthermore, $W(T) \subset \varphi_i^{-1}(0,1] \subset U_i$ for each $i \in T$ since $x \in W(T)$ implies $\varphi_i(x)>\varphi_j(x) \ge 0$ for each $i \in T$ and $j \in I \setminus T$. It follows that $W(T)$ is an ANE for $\M$ by Theorem~\ref{thm1} part \emph{i)}.

Note that if $\textrm{Card}(T)=\textrm{Card}(T')$ and $T \neq T'$, then $W(T) \cap W(T')=\emptyset$, since otherwise for some $x \in W(T) \cap W(T')$, $i \in T \setminus T'$ and $j \in T' \setminus T$ we have simultaneously $\varphi_i(x)<\varphi_j(x)$ and $\varphi_j(x)>\varphi_i(x)$, which is impossible.

Define 
$$
W_n=\bigcup \{W(T)|\textrm{Card}(T)=n\}.
$$
Then $W_n$ is an ANE for $\M$ by Theorem~\ref{thm1} part \emph{ii)}.

But now $X=\bigcup_{n \in \N} W_n$ is an ANE for $\M$ by Theorem~\ref{thm1}, part \emph{iii)}.
\end{proof}

\section{Acknowledgements}

The author wishes to thank the Magnus Ehrnrooth Foundation for financial support, and the Department of Mathematical Sciences at the University of Aarhus for its hospitality.

\end{document}